\documentclass[preprint]{elsarticle}
\usepackage{amsmath}
\usepackage{amsthm}
\usepackage{amssymb}
\usepackage{euscript}

\begin{document}
\title{There are Infinitely Many Perrin Pseudoprimes}
%\date{March 2009}
\author {Jon Grantham}
\address{Institute for Defense Analyses \\
Center for Computing Sciences\\
17100 Science Drive, Bowie, MD 20715, United States}
\ead{grantham@super.org}

\newtheorem{theorem}{Theorem}[section]
\newtheorem*{definition}{Definition}
\newtheorem*{definitions}{Definitions}
\newtheorem{proposition}[theorem]{Proposition}
\newtheorem{lemma}[theorem]{Lemma}

\begin{abstract}
This paper proves the existence of infinitely many Perrin pseudoprimes,
as conjectured by Adams and Shanks in 1982.  The theorem proven covers a 
general class of pseudoprimes based on recurrence sequences.  The result uses ingredients of the proof of the infinitude
of Carmichael numbers, along with zero-density estimates for Hecke
L-functions.
\end{abstract}
\def\gcmd{\operatorname{gcmd}}
\newcommand{\disc}{\operatorname{disc}}
\maketitle

\section{Introduction}

The search for fast primality tests has led to the examination of the generalization of the Fermat probable prime test: $n$ is a probable prime if $2^{n-1}\equiv 1 \bmod n$.  This test,
and its generalizations, requires $O(\log{n})$ multiplications. If such a generalization could be found with a finite list of exceptions (pseudoprimes), we would have a
primality test which runs deterministically in time $\operatorname{\tilde O}(\log^2{n})$.    (Recall that  $\operatorname{\tilde O}$ is an extension to the $\operatorname{O}$ notation that ignores factors that are bounded by a fixed power of the logarithm.)  By contrast, the Agrawal-Kayal-Saxena test \cite{aks} has recently been improved to $\operatorname{\tilde O}(\log^6{n})$ \cite{lenpom2}. Non-deterministic variants of the AKS test \cite{avanzi}, \cite{bernstein} have running time of $\operatorname{\tilde O}(\log^4{n})$; the same can be achieved heuristically for the ECPP test \cite{morain}.  Although ECPP is the fastest method in practice, it is not proven to be in (random) polynomial time.

The Fermat test can be generalized in many ways, which fall into two broad categories.  By thinking of it in terms of the first-order recurrence sequence defined by $a_{n+1}=2a_n$, $a_0=1$, we can generalize to congruences on higher-order recurrence sequences.  This approach is more traditional.  Alternatively, one can think of the Fermat criterion as the extent to which the ring of
integers mod $n$ resembles a finite field.  In that way, we can generalize to higher-degree finite fields.  The latter approach was favored in the author's dissertation \cite{dissme}, Chapter 4 of which contained an earlier version of the results of this paper.

In a 1982 paper \cite{as}, Adams and Shanks introduced a probable primality test based on third-order recurrence sequences, which they called the Perrin test. They asked if there are infinitely many Perrin pseudoprimes.
They answered the question,
``Almost certainly yes, but we cannot prove it.  Almost certainly, there
are infinitely many [Carmichael numbers which are Perrin pseudoprimes],
and yet it has never been proved that there are infinitely many Carmichael 
numbers.''

Carmichael numbers are composites which satisfy $a^{n-1} \equiv 1 \bmod n$ for all $(a,n)=1$.  The Carmichael question has been resolved \cite{agp}. 
The techniques of that proof can be combined with results about Hecke L-functions to
show that there are infinitely many Perrin pseudoprimes.  In fact, the main result of this
paper applies to a more general class of pseudoprimes, including Lucas and Lehmer
pseudoprimes.

\section{Background}
The following is a version of the Perrin test.

Consider sequences $A_n=A_n(r,s)$ defined by the following
relations:
$A_{-1}=s$, $A_0=3$, $A_1=r$, and $A_n=rA_{n-1}-sA_{n-2}+A_{n-3}$.
Let $f(x)=x^3-rx^2+sx-1$ be the associated polynomial and $\Delta$ its 
discriminant.
(Perrin's sequence is $A_n(0,-1)$.)

\begin{definition}
The {\bf signature mod $m$} of an integer $n$ with respect to the
sequence
$A_k(r,s)$ is the $6$-tuple
$(A_{-n-1},A_{-n},A_{-n+1},A_{n-1},A_n,A_{n+1}) \bmod m.$
\end{definition}

\begin{definitions}
An integer $n$ is said to have an {\bf S-signature} if its signature
mod $n$
is congruent to $(A_{-2},A_{-1},A_0,A_0,A_1,A_2)$.

An integer $n$ is said to have a {\bf Q-signature} if its signature
  mod $n$ is congruent to $(A,s,B,B,r,C)$, where for some integer $a$
with $f(a)\equiv 0 \bmod n$, $A\equiv a^{-2}+2a$, $B\equiv
-ra^2+(r^2-s)a$, and $C\equiv a^2+2a^{-1}$.

An integer $n$ is said to have an {\bf I-signature} if its signature
mod $n$ is congruent to $(r,s,D',D,r,s)$, where $D'+D\equiv rs-3$ mod
$n$ and $(D'-D)^2\equiv \Delta$.
\end{definitions}

\begin{definition}
A {\bf Perrin pseudoprime} with parameters $(r,s)$ is an odd composite
$n$ such that either

1) $\left(\frac{\Delta}n \right)=1$ and $n$ has an S-signature or an 
I-signature, or

2) $\left(\frac{\Delta}n \right)=-1$ and $n$ has a Q-signature.
\end{definition}

The concept of Perrin pseudoprime can be generalized \cite{meone} to that of a Frobenius pseudoprime.  Briefly,
a Frobenius pseudoprime with respect to $f(x)$ is a composite for which $\mathbb{Z}[x]/(n,f(x))$ exhibits properties
similar to that of a true finite field.  Most pseudoprime tests based on recurrence sequences can be treated as special cases.

\begin{definition}
Let $f(x)\in\mathbb{Z}[x]$ be a monic polynomial of degree $d$
with discriminant $\Delta$.   
An odd composite $n>1$ is said to be a {\bf Frobenius pseudoprime} with respect to $f(x)$ if $(n,f(0)\Delta)=1$,
and it is declared to be a probable prime by the following algorithm.
All computations are done in $(\mathbb{Z}/n\mathbb{Z})[x]$.

{\bf Factorization step} Let $f_0(x)=f(x) \bmod n$.
For $1\le i\le d$, let $F_i(x)=\gcmd(x^{n^i}-x,f_{i-1}(x))$ and
$f_i(x)=f_{i-1}(x)/F_i(x)$.
If any of the $\gcmd$s fail to exist, declare $n$ to be composite and stop.
If $f_d(x)\not=1$, declare $n$ to be composite and stop.

{\bf Frobenius step}  For $2\le i\le d$, compute $F_i(x^n) \bmod F_i(x)$.  If it is nonzero for some $i$, 
declare $n$ to be composite and stop.

{\bf Jacobi step} Let $S=\sum_{2|i}deg(F_i(x))/i$.

If $(-1)^S\not=\left(\Delta\over n\right)$,
declare $n$ to be composite and stop.

\vskip.05in

If $n$ has not been declared composite, declare $n$ to be a Frobenius probable prime.

(The $gcmd$ of two polynomials is the greatest common monic divisor; see \cite{meone} for a full treatment.)
\end{definition}

Haddad \cite{haddad} has shown that cubic variant of this test has running times which track well with asymptotics when implemented with Arjen Lenstra's Large Integer Package.

The following general result gives the infinitude of Perrin pseudoprimes as a corollary. 

\begin{theorem}\label{theorem:1}
Let $f(x)\in\mathbb{Z}[x]$
be a monic, squarefree polynomial with splitting field $K$.  There are
infinitely many Frobenius pseudoprimes with respect to $f(x)$.
In fact, there are $\gg N^c$ Carmichael-Frobenius numbers with respect
to $K$ which are less than $N$, for
some $c=c(K)>0$.
\end{theorem}

A Carmichael-Frobenius number is a Frobenius
pseudoprime with respect to all polynomials with splitting field $K$ (a Carmichael number is thus a Carmichael-Frobenius number with respect to $\mathbb{Q}$).
Proving the theorem for the general case allows
specialization to other cases.  The constant $c(K)$ can, in principle, be made effective.

In particular, the results of \cite{meone} combined with
Theorem~\ref{theorem:1}
show that there are infinitely many Perrin pseudoprimes,
if we take $f(x)=x^3-x-1$.  
Gurak \cite{gurak} defines pseudoprimes using congruences for higher-order recurrence sequences.  Szekeres \cite {sze} defines pseudoprimes with respect to a polynomial as those for which every symmetric polynomial of its roots is invariant under the map $x\mapsto x^n$.
From \cite{meone}, we have that there are
infinitely many pseudoprimes in the senses of both Gurak
and Szekeres.

By Proposition 6.1 of \cite{meone}, in order to prove Theorem~\ref{theorem:1},
it suffices to show that there are infinitely many
Carmichael numbers $n$, such that for all $p|n$, $f(x)$ splits completely mod $p$.
The proof will involve modifying the
construction in \cite{agp} to ensure that each of the prime factors
of the Carmichael numbers constructed has this property.
 
The main result that will be used in this proof is a version 
of the ``prime ideal theorem for arithmetic progressions'' that gives a
uniform error term, except for a possible exceptional progression arising from a Siegel zero.

\section{Distribution of Primes}
Theorems about the distribution of primes in arithmetic progressions
are traditionally proved using Dirichlet characters --- homomorphisms from the
integers mod $q$ to the complex roots of unity.  (The map is
defined to be zero on integers not coprime to $q$.)
Because we want to prove a theorem about primes in a particular
arithmetic progression which split completely, we
employ a slightly different sort of Dirichlet character.

We recall the definitions of \cite{lenpom}.

\begin{definitions}
Let $K$ be an algebraic number field and $\mathfrak{O}_K$ its ring of
integers.  A {\bf  cycle} of $K$ is a formal product $\mathfrak{m}=\prod
\mathfrak{p}^{m(\mathfrak{p})}$ extending over all of the primes of $K$, where
the $m(\mathfrak p)$ are nonnegative integers, almost all $0$, with
$m(\mathfrak{p})=0$ for complex $\mathfrak{p}$ and $m(\mathfrak{p})\le 1$ for real
$\mathfrak{p}$.  Let $\EuScript{I}$ be the group of fractional ideals of
$\mathfrak{O}_K$.  Let $\EuScript{I}(\mathfrak{m})$ be the subgroup of $\EuScript{I}$
generated by the finite primes $\mathfrak{p}$ for which $m(\mathfrak{p})=0$.
Let $P_{\mathfrak{m}}$ be the subgroup of $\EuScript{I}(\mathfrak{m})$ generated
by the nonzero ideals of the form $\mathfrak{O}_K\alpha$, where
$\alpha\in\mathfrak{O}_K$ is such that $\alpha\equiv 1 \bmod
\mathfrak{p}^{m(\mathfrak{p})}$ for each finite prime $\mathfrak p$, and
$\alpha>0$ under each embedding of $K$ in the field of real numbers
corresponding to a real prime $\mathfrak{p}$ with $m(\mathfrak{p})=1$.  The
{\bf norm} of a cycle $\mathfrak{m}=\prod\mathfrak{p}^{m(\mathfrak{p})}$ is the
number $N(\mathfrak{m})=\prod N(\mathfrak{p})^{m(\mathfrak{p})}$,
where $\mathfrak{p}$
in the second product ranges over only the finite primes, and
$N(\mathfrak{p})$ is the norm of $\mathfrak{p}$ in $K$.

A {\bf Dirichlet character of $K$} is a pair consisting of a cycle
$\mathfrak{m}$ of $K$ and a group homomorphism
$\chi:\EuScript{I}(\mathfrak{m})\mapsto\mathbb{C}^*$ such that
$P_{\mathfrak{m}}$
is contained in the kernel.  We call $\mathfrak{m}$ the {\bf modulus} of
$\chi$.

Given two Dirichlet characters $\chi$ and $\chi'$ with moduli $\mathfrak{m}$ 
and $\mathfrak{m}'$, we say that $\chi$ is {\bf induced} by $\chi'$ 
if $\mathfrak{m}'(\mathfrak{p})\le\mathfrak{m}(\mathfrak{p})$ for all
$\mathfrak{p}$ and $\chi$ 
is the composition of the inclusion 
$\EuScript{I}(\mathfrak{m})\subset\EuScript{I}(\mathfrak{m}')$ with $\chi'$.  A Dirichlet 
character is {\bf primitive} if it is not induced by any character other than 
itself.  The modulus of the unique primitive character inducing a 
Dirichlet character $\chi$ is called the {\bf conductor} of $\chi$.

For a Dirichlet character $\chi$ of $K$, $L(s,\chi)$ is $\sum
\chi(\mathfrak{i})N(\mathfrak{i})^{-s}$, where the sum is over the nonzero
ideals of the ring of integers of $K$ and $\operatorname{Re} s>1$.  This sum
is absolutely convergent, and $L(s,\chi')$ can be extended to a
meromorphic function on the complex plane.  It has a simple pole at
$s=1$ if $\chi'$ is principal and is holomorphic otherwise.

\end{definitions}

Let $K$ be the splitting field of $f$, $n=[K:\mathbb{Q}]$, and
$d=\disc(K)$.
Let $\chi$ be a Dirichlet character mod $q$ (in the traditional sense).  We associate to it a
Dirichlet character of $K$ in the
following way.

Given an ideal $\mathfrak{i}\subset\mathfrak{O}_K$, 
let $\chi'(\mathfrak{i})=\chi(N(\mathfrak{i}))$.
Then $\chi'$ is an example of a
Dirichlet character of $K$ with conductor dividing $N(q)$.

Let
$\Psi(x,\chi')=\sum_{N(\mathfrak{i})<x}\chi'(\mathfrak{i})\Lambda(\mathfrak{i})$,
where $\Lambda(\mathfrak{i})=\log N(\mathfrak{p})$
if $\mathfrak{i}=\mathfrak{p}^k$ for some prime ideal $\mathfrak{p}$, and $0$
otherwise.
Then $$\frac 1{\phi(q)}\sum_{\chi\bmod q} \bar\chi(a)\Psi(x,\chi')=
\sum_{\substack{
N(\mathfrak{i})<x\\N(\mathfrak{i})\equiv a\bmod q}} \Lambda(\mathfrak{i}).$$

We will prove some results about $L(s,\chi')$ that enable us to obtain
results about $\Psi(x,\chi')$, and thus about the distribution of
primes that split completely in $K$ and lie in a particular residue
class.

\begin{lemma}
Fix a number field $K$.
Let $\chi$ be a {\bf real} Dirichlet character of $K \bmod \mathfrak{m}$.
Let $M=N(\mathfrak{m})$.
Let $s$ be a real number in the range $2>s>1$.
If $\chi$ is principal, then
$$\frac {L'}L(s,\chi) > -\frac 1{s-1}-c_1\log 2M,$$
for some $c_1>0$, depending on $K$.
If $\chi$ is non-principal, and if $L(s,\chi)$ has some real zero $\rho>0$,
$$\frac {L'}L(s,\chi) > \frac 1{s-\rho}-c_1\log 2M,$$
and 
$$\frac {L'}L(s,\chi)>-c_1\log 2M,$$ if it has no real zero.
%(The `$2M$' is necessary so $\log 2M>0$ for $M=1$.
\end{lemma}
\begin{proof}
Assume $\chi$ is non-principal.  From Eq. (5.9) of \cite{lo},
$$\frac {L'}L(s,\chi)=B(\chi)+
\sum_{\rho}\left(\frac 1{s-\rho}+\frac 1\rho\right)
-\frac 12\log A(\chi)-\frac{\gamma'_\chi}{\gamma_\chi}(s),$$
where the sum is over all the non-trivial zeroes of $L(s,\chi)$,
$A(\chi)=dM$, $d=\disc(K)$, and $\gamma_\chi(s)=
\left[\pi^{-\frac{s+1}2}\Gamma\left(\frac{s+1}2\right)\right]^b
\left[\pi^{-\frac s2}\Gamma\left(\frac s2\right)\right]^a$ for
nonnegative integers $a$ and $b$ depending on $\chi$ such that
$a+b=n=[K:\mathbb{Q}]$.
The exact dependence, described in \cite{lo}, is irrelevant here.

The $\log A(\chi)$ term can be bounded because $\log A(\chi)=\log
dM\ll \log 2M$.

$B(\chi)$ is defined implicitly in \cite{lo}.
By Lemma 5.1 of that paper, we have
$$B(\chi)=-\operatorname{Re}\sum_{\rho}\frac 1{\rho}.$$

We have from Lemma 5.3 of \cite{lo} that
$$|\frac{\gamma'_\chi}{\gamma_\chi}(s)|\ll n\log(s+2),$$
where the implied constant is absolute.

Thus $\frac{L'}L(s,\chi)>\sum_{\rho}\operatorname{Re}\frac
1{s-\rho}-c_1\log 2M$, for some $c_1>0$.
We have that $\operatorname{Re}\frac 1{s-\rho}=\frac {s-\operatorname{Re}\rho}
{|s-\rho|^2}>0$, so we 
can omit any part of the sum.  We omit anything but one possible real zero.

Thus
$$\frac{L'}L(s,\chi)>\frac 1{s-\rho}-c_1\log 2M,$$
if $L(s,\chi)$ has a real zero $\rho$, and
$$\frac{L'}L(s,\chi)>-c_1\log 2M,$$
independent of the existence of real zeros.

Now assume $\chi$ is principal.
From (5.9) of \cite{lo}
$$\frac {L'}L(s,\chi)=\sum_{\rho}\left(\frac 1{s-\rho}+\frac 1\rho\right)
-\frac 1s-\frac 1{s-1}-\frac
12\log
A(\chi) +\frac{\gamma'_\chi}{\gamma_\chi}(s).$$
By the same arguments as in the non-principal case (and the fact that
$\frac 1s<1$), we have that
$$\frac{L'}L(s,\chi)>-\frac 1{s-1}-c_1\log 2M.$$
\end{proof}

The following version of the Landau-Page Lemma for Dirichlet $L$-functions over a number field shows that there is at
most one ``Siegel zero'' for characters of a bounded modulus.

\begin{lemma}
Given a number field $K$,
there is a computable constant $c_2>0$, depending on $K$,
such that for all $T\ge 2$,
there is at most one primitive character $\chi_1$ with modulus
$\mathfrak{m}$, $1\le N(\mathfrak{m})<
T$ for which $L(s,\chi_1)$ has a zero $\beta_1+i\gamma_1$ satisfying
$\beta_1\ge 1-c_2/\log T$ and $|\gamma_1|<T$.
\end{lemma}
\begin{proof}
We follow the proof in [6, p. 94].

%Lemma 2.3 of \cite{lmo} (which also appears as
%Lemma 3.5 of \cite{\lenpom}) allows us
%to consider only real zeros of real non-principal characters.

Lemma 3.5 of \cite{lenpom} allows us
to consider only real zeros of real non-principal characters.

Let $\chi_1$ and $\chi_2$ be primitive characters mod $\mathfrak{m}_1$
and $\mathfrak{m}_2$,
respectively, where $N(\mathfrak{m}_1)$ and $N(\mathfrak{m}_2)$ are at most $T$.

Consider the expression
$$-\frac{L'}L(s,\chi_0)-\frac{L'}L(s,\chi_1)-\frac{L'}L(s,\chi_2)
-\frac{L'}L(s,\chi_1\chi_2),$$
where $\chi_0$ is the principal character modulo the $\gcd$ of $\mathfrak{m}_1$
and $\mathfrak{m}_2$.  (We define
$\gcd(\prod\mathfrak{p}^{m_1(\mathfrak{p})},\prod\mathfrak{p}^{m_2(\mathfrak{p})})
=\prod\mathfrak{p}^{\min(m_1(\mathfrak{p}),m_2(\mathfrak{p}))}$.)
This expression is equal to 
\begin{equation}\sum
\Lambda(\mathfrak{i})(\chi_0(\mathfrak{i})+\chi_1(\mathfrak{i}))(\chi_0(\mathfrak{i})+\chi_2(\mathfrak{i}))N({\mathfrak i})^{-s}>0,\end{equation}
for $\operatorname{Re} s>1$.

Assume that $L(s,\chi_1)$ and $L(s,\chi_2)$ have real zeros, $\beta_1$ and
$\beta_2$ respectively.
Applying the previous lemma to \thetag{1} for real $s>1$, we obtain
$$-\frac{1}{s-\beta_1}-\frac{1}{s-\beta_2}+\frac
1{s-1}+c_3\log T>0,$$
for some $c_3>0$ depending on $K$, but not $T$.
Rearranging, we get
$$\frac 1{s-\beta_1}+\frac 1{s-\beta_2}<\frac
1{s-1}+c_3\log T.$$

Let $c_2=\frac 1{6c_3}$ and assume that each $\beta_i\ge 1-\frac{c_2}{\log T}$.

Taking $s=1+3c_2/\log T$ gives us
$\frac 1{s-\beta_i}\ge\frac{\log T}{4c_2}$ and $\frac 1{s-1}=
\frac{\log T}{3c_2}$.

We now have that
$$\frac{\log T}{2c_2}<\frac{\log T}{3c_2}+c_3\log T.$$
Simplifying, we get $\frac 1{6c_2}<c_3$.  Substituting the value of $c_2$ gives the
desired contradiction.
\end{proof}

For each Dirichlet character $\chi$ of a field $K$ and real numbers
$\sigma$, $T$, in the ranges $\frac 12\le\sigma\le 1$, $T\ge 0$, let
$N(\sigma,T,\chi)$ be the number of zeros $s=\beta+i\gamma$ of the
Dirichlet L-function $L(s,\chi)$ inside the box $\sigma<\beta< 1$
and $|\gamma|< T$.
Let $\EuScript{A}$ be the set of real numbers $A>2$ for which there exists
a number $C_A\ge 1$, such that for all $\sigma\ge 1-\frac 1A$ and
$T\ge 1$, we have
$$N(\sigma,T,\mathfrak{m}):=\sum_{\chi\bmod \mathfrak{m}} N(\sigma,T,\chi)\le 
C_A(N(\mathfrak{m})T^n)^{A(1-\sigma)},$$
for all moduli $\mathfrak m$.

Hilano \cite{hiltwo} has shown that every $A\ge 2890$ is in $\EuScript{A}$.
The existence of such an $A$ was first shown by Fogels \cite{fogels}.

\begin{theorem}
Let $K$ be a number field.
For any given $\epsilon>0$, there exist
numbers $x_{\epsilon}$,
$\eta_{\epsilon}>0$, and an integer $q_\epsilon(x)$, all depending on
$K$,
 such that whenever $x\ge
x_{\epsilon}$ and $x^{1/2}<y<x$,
$$\left|\sum_{\substack{
N(\mathfrak{i})<y\\N(\mathfrak{i})\equiv a\bmod q}} \Lambda(\mathfrak{i})-
\frac{y}{\phi(q)}\right|\le\epsilon\frac{y}{\phi(q)}$$
for all integers $q$ not divisible by $q_\epsilon(x)$,
with $(a,q)=1$ and $q$ in the range $1\le q\le x^{\eta_{\epsilon}}$.
Furthermore $q_\epsilon(x)>\log x$.
\end{theorem}
\begin{proof}
Let $\nu=3\log(36C_A/\epsilon)$.
Let $\eta_\epsilon=\min(\frac 1{8An^2},\frac{c_2}{n\nu})$.
We can require $x_\epsilon>\max(e^{4A\nu/\eta_\epsilon},
18(C_A/\epsilon)^{2/\eta_\epsilon})$.

We can deduce the following explicit formula from \cite{lenpom},
proof of Theorem 3.1:
(Eqs. (3.2), (3.3) and the equation following the ``Hence'' on p. 493).

\begin{equation}
\begin{split}
\sum_{\substack{
N(\mathfrak{a})<y\\N(\mathfrak{a})\equiv a\bmod q}}
\Lambda(\mathfrak{a})=\frac{y}{\phi(q)}-\frac 1{\phi(q)}
\sum_{\chi\bmod q}\bar\chi(a)&\sum_{\substack{{L(\beta+i\gamma,\chi)=0}\\
\beta\ge 1/2, |\gamma|\le T}} \frac
{y^{\beta+i\gamma}}{\beta+i\gamma}+\\
O\Bigl(n\log y+
n\frac{y\log y(\log y+\log dq+\log T)}T+&\\
n\log dq+ny^{\frac 12}\log y&(\log q+\log y)\Bigr).
\end{split}
\end{equation}

We have that $\eta_\epsilon<1/16$ (since by definition, $A>2$), so $\log q<\log y$ and $q<y^{1/3}$.
We take $T=x$, so $y<T<y^2$.

The error term in $\thetag{2}$ is
$$O\left(ny^{1/2}(\log^2 y+\log y\log d)+n\log d\right).$$
Because $d,n$ are fixed, the error is 
$O(y^{1/2}\log^2y)=O(\frac{y^{6/7}}q)$, which is less than
$\frac{\epsilon}3\frac y{\phi(q)}$ for $y$ sufficiently large.

The double sum may be bounded by noting that $|y^{\beta+i\gamma}|=y^\beta$,
and $\beta+i\gamma\ge\sqrt{1/4+\gamma^2}\ge (1+|\gamma|)/3$.

Thus
\begin{equation}
\left|\sum_{\substack{
N(\mathfrak{a})<y\\N(\mathfrak{a})\equiv 1\bmod q}}
\Lambda(\mathfrak{a})-\frac{y}{\phi(q)}\right|\le\frac 3{\phi(q)}
\sum_{\chi\bmod q}\sum_{\substack{{L(\beta+i\gamma,\chi)=0}\\\
\beta\ge 1/2, |\gamma|\le x}}\frac{y^\beta}{1+|\gamma|}
+\frac\epsilon 3\frac{y}{\phi(q)}.\end{equation}

Write $\sum_\sigma^{\alpha}$ for a sum over all zeros of
$\beta+i\gamma$ of $L(s,\chi)$ and over all characters $\chi \bmod q$
where $\sigma\le\beta<\alpha$ and $|\gamma|<x$.  (Each $\beta+i\gamma$
is counted with multiplicity equal to the number of those L-functions
for which it is a zero.)  To estimate the double sum in \thetag{3}
we use the upper bounds $y^\beta\le y^{1-1/(2An)}$ for $\beta\le
1-1/(2An)$, and $y^\beta\le y$ for $\tau\le\beta\le 1$, where $\tau=1-\nu/\log
x$.  In the range
$1-1/(2An)\le\beta\le\tau$, we use the identity $y^\beta=y^{1-1/(2An)}+\log y
\int_{1-1/(2An)}^\beta y^\sigma d\sigma$.

Therefore, the double sum in \thetag{3} is at most
\begin{equation}
\begin{split}
&\sum_{1/2}^{1-1/(2An)} \frac{y^{1-1/(2An)}}{1+|\gamma|}+\log y\sum_{1-1/(2An)}^\tau
\frac 1{1+|\gamma|}\int_{1-1/(2An)}^\beta y^{\sigma} d\sigma+y
\sum_\tau^1 \frac 1{1+|\gamma|}\\
&\le y^{1-1/(2An)}\sum^1_{1/2}\frac 1{1+|\gamma|}+\log
y\int_{1-1/(2An)}^\tau y^\sigma\left(\sum_\sigma^1 \frac 1{1+|\gamma|}\right)d\sigma\\
&+y\sum_\tau^1 \frac 1{1+|\gamma|}.
\end{split}
\end{equation}
%Equation 4

For $\sigma\ge 1/2$, we have, by partial summation,
$$\sum_\sigma^1\frac 1{1+|\gamma|}\le N(\sigma,2,q)+\frac{N(\sigma,x,q)}x
+\int_2^{x}\frac{N(\sigma,t,q)}{t^2} dt.$$

By \cite{hilone}, %I could use a better citation for this.
$N(1/2,t,q)<c_4nq^nt\log qt$.  For $t$ in the
range $2\le t\le x$, we have $N(1/2,t,q)/t\le c_4nq^n\log qx$.

Applying this,
$$\sum_{1/2}^1 \frac{y^{1-\frac1{2An}}}{1+|\gamma|}\le
2c_4nq^ny^{1-\frac1{2An}}\log qx\left(2+\int_2^x\frac{dt}t\right)\le
5c_4nq^ny^{1-\frac1{2An}}\log^2 qx.$$
Because we insist that $\eta_\epsilon<\frac 1{8An^2}$, the first term in
\thetag{4} is
$O(y^{1-1/(3An)})$, which is $<\frac{\epsilon}{18}y$ for $y$
sufficiently large.

If $\sigma\ge 1-1/(2An)$, then $An(1-\sigma)\le 1/2$, so that for any $t$ in
the range $1\le t\le x$, Theorem 9 of \cite{hiltwo} shows that
$N(\sigma,t,d)\le C_Aq^{An(1-\sigma)}t^{1/2}$.  We deduce that
$$\sum_\sigma^1\frac 1{1+|\gamma|} \le C_A q^{An(1-\sigma)}
\left(3+\int_2^{x}\frac{dt}{t^{3/2}}\right)\le 5C_Aq^{An(1-\sigma)}.$$

Using this bound, the middle term in \thetag{4} is 
\begin{equation}
\begin{split}
&\le 5C_Aq^{An}\log y\int_{1-1/(2An)}^\tau \left(\frac
y{q^{An}}\right)^\sigma d\sigma\\
&\le 5C_Aq^{An}\frac{\log y}{log(y/q^{An})}\frac y{q^{An}}
\left(\frac y{q^{An}}\right)^{-(1-\tau)}.
\end{split}
\end{equation}
%Equation 5

We have that 
$q^{An}<x^{An\eta}<y^{1/3}$, so $\frac{\log y}{\log (y/q^{An})}< 3/2$.
Also, $$\left(\frac y{q^{An}}\right)^{-(1-\tau)}=\left(\frac
y{q^{An}}\right)^{-\nu/\log x}<e^{-\frac 23\log y\nu/\log
x}<e^{\frac 13\nu}.$$

Thus the middle term in \thetag{4} is $\le 4C_Aye^{-\frac 13\nu}$, which, by the way we chose $\nu$, is $\le\frac{\epsilon}9 y$.

We apply Lemma 2.2 with $T=x^{n\eta_\epsilon}$ and call
the exceptional modulus $q_\epsilon(x)$.  Then for all moduli
less than $x^{\eta_\epsilon}$ and not divisible by $q_\epsilon(x)$, the
$L$-function has no zeros $\beta+i\gamma$ with 
$\beta\ge\tau=1-\nu/\log x$ and $|\gamma|<x^{n\eta_\epsilon}$.

So the third term in \thetag{4} is
$$y\sum^1_\tau \frac 1{1+|\gamma|}\le y\frac{N(\tau,x,q)}{x^{n\eta_\epsilon}}
\le C_Ay(q^nx^n)^{A(1-\tau)}/x^{n\eta_\epsilon} < C_Ayx^{2An\nu/\log x}/x^{n\eta_\epsilon}.$$
This is less than
$C_Ayx^{-\eta_\epsilon/2}$, by our choice of $x$.
Also, since $x>x_\epsilon$, by our choice of $x$, this is less  than
$C_Ay{(18C_A/\epsilon)^{2/\eta_\epsilon}}^{-\eta_\epsilon/2}=\epsilon y/{18}$.
Putting these bounds together, we get the desired theorem.
\end{proof}

Theorem 2.1 of {\cite{agp}} shows, essentially, that the number of primes in an
arithmetic progression less than $x$ cannot be too far away from what you expect.
Furthermore, it shows this for ``most'' moduli up to $x^{\frac 5{12}}$.
Our replacement is the following

\begin{theorem}
Let $f(t)\in\mathbb{Z}[t]$ be a monic polynomial with splitting field
$K$, $[K:\mathbb{Q}]=n$.
Then we have real numbers $x_{1/3},\eta_{1/3}>0$
and an integer $q_{1/3}(x)>\log
x$, depending on $K$ as described in Theorem 3.3, such that the
following statement holds.
If $q\le x^{\eta_{1/3}}$, $\gcd(a,q)=1$, $q_{1/3}(x)\nmid q$, $x\ge
x_{1/3}$ and $x^{1/2}<y<x$, then the number of primes $p<y$
that are $a \bmod q$ and such that  $f(t)$ splits into linear factors
mod $p$
(equivalently, $p$ splits completely in $K$)
is at least $\frac{1}{2\phi(q)n}\pi(x)$.  
\end{theorem}
\begin{proof}
The previous theorem gives that
$$\sum_{\substack{
N(\mathfrak{a})<y\\N(\mathfrak{a})\equiv a\bmod q}}
\Lambda(\mathfrak{a})\ge
\frac{(2/3)y}{\phi(q)}.$$

The sum contains two types of summands not arising from primes.  The first,
prime ideal powers, can be dispensed of in the usual way, by noting that their contribution to the sum
is $O(y^{1/2})$.  The second type is primes that do not split
completely, for which we have $N(p)=p^k$, for $k>1$, so they also contribute $O(y^{1/2})$. 
We pass to the estimate on the number of primes by standard techniques
(\cite{dav}, p. 112).
\end{proof}

Henceforth, let $\eta=\eta_{1/3}$ and $q(x)=q_{1/3}(x)$.

\section{Prachar's Theorem}
We use the following variant of Prachar's Theorem (c.f. Theorem 3.1 of
\cite{agp}).

\begin{theorem}
If $L$ is a squarefree number not divisible
by any prime exceeding $x^{\frac{1-\eta}{2}}$ and for which
$\sum_{\text{prime}\ q|L}\frac 1q\le\frac{1-\eta}{32n}$, then there is a
positive integer $k\le x^{1-\eta}$ with $(k,L)=1$ such that
$$\#\{d|L:dk+1\le x, dk+1 \text{ is prime, splits fully in } K\}\ge \frac
{\#\{d|L: 1\le d\le x^{\eta}\}}{8n\log x}.$$
\end{theorem}
\begin{proof}
Let $\pi_K(x;q)$ denote the number of primes less than $x$
that are $1 \bmod q$ and split completely in $K$.

From Theorem 3.4, we see that for each divisor $d$ of $L$ with $1\le d\le
\smash{x^{\eta}}$ and $(d,q(x))=1$,
$$\pi_K(dx^{1-\eta};d)\ge \frac{\pi(dx^{1-\eta})}{2n\phi(d)}\ge
\frac{dx^{1-\eta}}{2n\phi(d)\log x}.$$

Because any prime factor $q$ of $L$ is at most $x^{\frac{1-\eta}2}$, we
can use Montgomery and Vaughan's explicit version of the
Brun-Titchmarsh Theorem \cite{MV} to get 
$$\pi_K(dx^{1-\eta};dq)\le\pi(dx^{1-\eta};dq,1)\le\frac 8{q(1-\eta)}
\frac{dx^{1-\eta}}{\phi(d)\log{x}}.$$

So the number of primes $p\le dx^{1-\eta}$ with $p\equiv 1\bmod d$
and $(\frac{p-1}d,L)=1$ that split completely is at least
$$\pi_K(dx^{1-\eta};d)-\sum_{\text{prime } q|L}\pi_K(dx^{1-\eta};dq)$$
$$\ge\left(\frac 1{2n}-\frac 8{1-\eta}\sum_{\text{prime }q|L}\frac 1q\right)
\frac{dx^{1-\eta}}{\phi(d)\log{x}}\ge\frac {x^{1-\eta}}{4n\log
x},$$ for any divisor not divisible by $q(x)$.  But at least half of the
divisors of $L$ will not be divisible by $q(x)$.

Thus we have at least
$$\frac{x^{1-\eta}}{8n\log x}\#\{d|L:1\le d\le x^{\eta}\}$$
pairs $(p,d)$ where $p\le d^{1-\eta}$ is prime, $p\equiv 1\bmod
d$, $p$ splits completely in $K$, $(\frac{p-1}d,L)=1$, $d|L$ and $1\le d\le
x^\eta$.  Each such pair $(p,d)$ corresponds to an integer $\frac
{p-1}d\le x^{1-\eta}$ which is coprime to $L$, so there is at least
one integer $k\le x^{1-\eta}$ with $(k,L)=1$ such that $k$ has at least 
$$\frac 1{8n\log x}\#\{d|L: 1\le d\le x^\eta\}$$
representations as $\frac{p-1}d$ with $(p,d)$ as above. Thus for this integer $k$ we have
$\#\{d|L:dk+1\le x, dk+1\text{ prime, split completely in $K$}\}\ge \frac 1{8n\log x}\#\{d|L: 1\le d\le x^\eta\}$.
\end{proof}

\section{Infinitely Many Frobenius Pseudoprimes}
We recall the results from Section 1 of \cite{agp}.

\notag\begin{theorem}[\text{See [3, Theorem 1.1].}]
Let $n(G)$ be the length of the longest sequence of (not necessarily distinct) elements
of $G$ for which no non-empty subsequence has product the identity. If $G$ is a finite abelian group and $m$ is the maximal order of an
element in $G$, then $n(G)<m(1+\log{(\frac{|G|}m)})$.  
\end{theorem}

This theorem is due to van Emde Boas and Kruyswijk,
and to Meshulam.

\notag\begin{proposition}[\text{See [3, Proposition 1.2].}]
Let G be a finite abelian group, and let $r>t>n=n(G)$ be integers.
Then any sequence of $r$ elements of G contains at least
$\dfrac{\binom{r}t}{\binom{r}n}$ distinct
subsequences of length at most $t$ and at least $t-n$, whose product
is the identity.
\end{proposition}

We now prove our main result, which was stated earlier as Theorem 2.1.

\begin{theorem}
Let $K$ be a number field, and let $\eta$ be the positive real number
depending on $K$ defined in Theorem 3.3.
For any $\epsilon>0$,
the number of Carmichael-Frobenius numbers less than $x$,
with respect to a number field $K$,
is at least $x^{\eta/3-\epsilon}$, for
sufficiently large $x$, depending on $\epsilon$ and $K$.
\end{theorem}
\begin{proof}
Let $\EuScript Q$ be the set of primes $q\in(\frac{y^3}{\log y},y^3]$ for
which $q-1$ is free of prime factors exceeding $y$.  
Friedlander \cite{Fri} has proven
that there is a constant $C>0$  for which 
$$|\EuScript Q|\ge C\frac{y^3}{\log{y}}$$ for all sufficiently large $y$.
Let $L$ be the product of the primes $q\in \EuScript Q$; then
$$\log L\le|\EuScript Q|\log{(y^3)}\le\pi(y^3)\log{(y^3)}\le 2y^3,$$ for
all large $y$.  Carmichael's lambda function, $\lambda(L)$, is the exponent of
the group of integers modulo $L$.  Because $L$ is squarefree, it is the least common multiple of 
$\{q-1\}$ for those primes $q$ that divide $L$.  Because each such
$q-1$ is free of prime factors exceeding $y$, we know that if the
prime power $p^a$ divides $\lambda(L)$ then $p\le y$ and $p^a\le
y^3$.  We let
$p^{a_p}$ be the largest power of $p$ with $p^{a_p}\le y^3$, then
$$\lambda(L)\le \prod_{p\le y} p^{a_p}\le\prod_{p\le y} y^3 = 
y^{3\pi(y)}\le e^{6y}$$ for all large $y$.

Let $G$ be the group $(\mathbb{Z}/L\mathbb{Z})^*$.  From Theorem 5.1 and the
above equations,
$$n(G)<\lambda(L)\left(1+\log{\frac{\phi(L)}{\lambda(L)}}\right)\le
\lambda(L)(1+\log L)\le e^{9y}$$ for all large $y$.

Recall that $\eta<1/16$. We can choose $y$ large enough so that $\sum \frac 1q\le
\frac{1-\eta}{32n}$ as needed to apply Theorem
4.1.  Let $\delta=\frac{3\epsilon}{8n\eta}$, and let
$x=e^{y^{1+\delta}}$.
Then, for $y$ large enough, there is an integer $k$ coprime to
$L$ for which the set $\EuScript P$ of primes $p\le x$ with $p=dk+1$ for
some divisor $d$ of $L$, and that split in $K$, satisfies 
$$|\EuScript P|\ge\frac{\#\{d|L:1\le d\le x^{\eta}\}}{8n\log x}.$$
The product of any 
$$u:=\left[\frac{\log\left(x^{\eta}\right)}{\log{y^3}}\right]=
\left[\frac{\eta\log x}{3\log y}\right]$$
distinct prime factors of $L$ is a divisor $d$ of $L$ with $d\le
x^\eta$.  We deduce from above that
$$\#\{d|L:1\le d\le x^\eta\}\ge{\binom{\omega(L)}u}\ge\left(\frac
{\omega(L)}{u}\right)^u$$
$$\ge\left(\frac{Cy^3}{\eta\log x}\right)^u=\left(\frac
C\eta y^{2-\delta}\right)^u.$$

We notice that $\frac{(2-\delta)\eta}3=\frac{2\eta}3-\frac\epsilon {8n}$.
So for all sufficiently large values of $y$,
$$|\EuScript P|\ge\frac{\left(\frac C\eta y^{2-\delta}\right)^u}{8n\log x}\ge
x^{\frac{2\eta}3-\frac\epsilon 3}.$$

Take $\EuScript P'=\EuScript P\backslash\EuScript Q$.  Because $|\EuScript Q|\le y^3$,
we have that $|\EuScript P'|\ge x^{\frac{2\eta}3-\frac\epsilon 2}$, for all
sufficiently large values of $y$.  

We may view $\EuScript P'$ as a subset of the group
$G=(\mathbb{Z}/L\mathbb{Z})^*$ by considering the residue class of each
$p\in\EuScript P' \bmod L$.  If $\EuScript S$ is a subset of $\EuScript P'$
that contains more than one element, and if
$$\prod(\EuScript S):=\prod_{p\in\EuScript S}p\equiv 1\bmod L,$$
then $\prod(\EuScript S)$ is congruent to $1\bmod kL$ and is a Carmichael number by Korselt's
criterion.  Because all of its prime factors split completely in $K$, it is a Frobenius pseudoprime.

Let $t=e^{y^{\frac{1+\delta}2}}$.  Then, by Proposition 5.2, we see
that the number of Frobenius pseudoprimes of the form $\prod(\EuScript S)$,
where $\EuScript S\subset\EuScript P'$ and $|\EuScript S|\le t$, is at least 
$$\frac{{\binom{|\EuScript P'|}{[t]}}}{{\binom{|\EuScript P'|}{n(G)}}} \ge
\frac{\left(\frac{|\EuScript P'|}{[t]}\right)^{[t]}}{|\EuScript P'|^{n(G)}} \ge
\left(x^{\frac{2\eta}3-\frac\epsilon 2}\right)^{[t]-n(G)}[t]^{-[t]} \ge
x^{t\left(\frac{2\eta}3-\epsilon\right)}$$ for all sufficiently large
values of $y$.  We note that we have formed each Frobenius
pseudoprime
$\prod(\EuScript P)\le x^t$.  Thus for $X=x^t$ we have the number of
Frobenius pseudoprimes $\le x$ is at least
$X^{\frac{2\eta}3-\epsilon}$ for all sufficiently large values of $X$.
Because $y$ can be uniquely determined from $X$, the theorem is proven.

\end{proof}

\bibliography{ps3}

\begin{thebibliography}{10}

\bibitem{as}
W.~W. Adams and D.~Shanks.
\newblock Strong primality tests that are not sufficient.
\newblock {\em Math. Comp.}, 39:255--300, 1982.

\bibitem{aks}
Manindra Agrawal, Neeraj Kayal, and Nitin Saxena.
\newblock {PRIMES} is in {P}.
\newblock {\em Ann. of Math.}, 160:781--793, 2004.

\bibitem{agp}
W.~R. Alford, Andrew Granville, and Carl Pomerance.
\newblock There are infinitely many {C}armichael numbers.
\newblock {\em Annals of Mathematics}, 140:703--722, 1994.

\bibitem{avanzi}
Roberto~M. Avanzi and Preda Mih{\u{a}}ilescu.
\newblock Efficient ``quasi''-deterministic primality test improving {AKS}.
\newblock http://caccioppoli.mac.rub.de/website/papers/aks-mab.pdf, 2009.

\bibitem{bernstein}
Daniel~J. Bernstein.
\newblock Proving primality in essentially quartic random time.
\newblock {\em Math. Comp.}, 76:389--403, 2007.

\bibitem{dav}
Harold Davenport.
\newblock {\em Multiplicative Number Theory}.
\newblock Springer-Verlag, New York, second edition edition, 1980.

\bibitem{fogels}
E.~Fogels.
\newblock On the zeros of {$L$}-functions.
\newblock {\em Acta Arith.}, 11:67--96, 1965.

\bibitem{Fri}
J.~B. Friedlander.
\newblock Shifted primes without large prime factors.
\newblock In {\em Number theory and applications ({B}anff, {AB}, 1988)}, volume
  265 of {\em NATO Adv. Sci. Inst. Ser. C Math. Phys. Sci.}, pages 393--401.
  Kluwer Acad. Publ., Dordrecht, 1989.

\bibitem{dissme}
J.~Grantham.
\newblock {\em Frobenius Pseudoprimes}.
\newblock PhD thesis, University of Georgia, 1997.

\bibitem{meone}
J.~Grantham.
\newblock Frobenius pseudoprimes.
\newblock {\em Math. Comp.}, 70:873--891, 2001.

\bibitem{gurak}
S.~Gurak.
\newblock Pseudoprimes for higher-order linear recurrence sequences.
\newblock {\em Math. Comp.}, 55:783--813, 1990.

\bibitem{haddad}
Jihad~Michael Haddad.
\newblock A comparison of the frobenius primality test with the strong
  primality test.
\newblock Technical report, University of Odense, 1998.
\newblock Bachelor Project.

\bibitem{hilone}
Teluhiko Hilano.
\newblock On the zeros of {H}ecke's {$L$}-functions.
\newblock {\em Sci. Papers College Gen. Ed. Univ. Tokyo}, 24:9--24, 1974.

\bibitem{hiltwo}
Teluhiko Hilano.
\newblock On the zeros of {Heck's} {\rm [sic]} {$L$}-functions.
\newblock {\em Proc. Japan Acad.}, 50:23--28, 1974.

\bibitem{lenpom}
H.W.~Lenstra Jr. and Carl Pomerance.
\newblock A rigorous time bound for factoring numbers.
\newblock {\em J. Amer. Math. Soc.}, 5:483--516, 1992.

\bibitem{lenpom2}
H.W.~Lenstra Jr. and Carl Pomerance.
\newblock Primality testing with gaussian periods.
\newblock http://math.dartmouth.edu/\~{ }carlp/aks0221109.pdf, 2009.

\bibitem{lo}
J.C. Lagarias and A.M. Odlyzko.
\newblock Effective versions of the chebotarev density theorem.
\newblock In A.~Frolich, editor, {\em Algebraic Number Fields: $L$-functions
  and Galois Properties}, pages 409--464. Academic Press, London, 1977.

\bibitem{MV}
H.~L. Montgomery and R.~C. Vaughan.
\newblock Error terms in additive prime number theory.
\newblock {\em Quart. J. Math. Oxford Ser. (2)}, 24:207--216, 1973.

\bibitem{morain}
F.~Morain.
\newblock Implementing the asymptotically fast version of the elliptic curve
  primality proving algorithm.
\newblock {\em Math. Comp.}, 76:493--505, 2007.

\bibitem{sze}
G.~Szekeres.
\newblock Higher order pseudoprimes in primality testing.
\newblock In {\em Combinatorics, Paul Erd\H os is eighty}, volume~2, pages
  451--458. J\'anos Bolyai Math Soc., Budapest, 1996.

\end{thebibliography}
\bibliographystyle{plain}
\end{document}